\documentclass[a4paper,12pt]{article}

\usepackage[usenames,dvipsnames]{xcolor}
\usepackage{amssymb}
\usepackage{amsfonts}
\usepackage{tabulary}
\usepackage{amsmath}
\usepackage{textcomp}
\usepackage{setspace}
\usepackage{graphicx}
\usepackage{authblk}
\newtheorem{theorem}{Theorem}
\newtheorem{definition}[theorem]{Definition}
\newtheorem{lemma}[theorem]{Lemma}
\newtheorem{example}[theorem]{Example}
\newtheorem{proposition}[theorem]{Proposition}

\newtheorem{remark}[theorem]{Remark}

\newenvironment{proof}[1][Proof]{\noindent\textbf{#1. }}{\ $\square$\bigskip}

\begin{document}

\title{Killing Vectors and Bochner-Type Theorems in Pseudo-Riemannian Hom-Lie Algebras}

\author{Raymond HOUNNONKPE \and Ibrahima BAKAYOKO \and Emmanuel GNANDI \and Jo\"el TOSSA }
\maketitle

\vskip 0.2truecm

\begin{abstract}
In this paper, we study pseudo-Riemannian Hom-Lie algebras. We establish Hom-versions of the classical Bochner theorem in both the Riemannian and Lorentzian settings. Finally, we prove that the space of Killing vectors on a pseudo-Riemannian Hom-Lie algebra forms a totally geodesic Hom-Lie subalgebra.
\end{abstract}

\begingroup
\renewcommand\thefootnote{}
\footnotetext{
\textbf{2020 Mathematics Subject Classification.} Primary 53A15, 53C15, 53B30; Secondary 53C55, 53D05.

\textbf{Keywords.} Hom-Lie algebras, pseudo-Riemannian Hom-Lie algebras, Bochner theorem, Killing vectors.
}
\endgroup

\section{Introduction}

In Riemannian geometry, the classical Bochner theorem \cite{Boch} states that on
a compact Riemannian manifold with non-positive Ricci curvature, any Killing
vector field is parallel; moreover, if the Ricci curvature is strictly negative,
no non-trivial Killing vector field exists.

The notion of Hom-Lie algebra was introduced by Hartwig, Larsson, and Silvestrov
\cite{Hartwig2006} as a generalization of Lie algebras, motivated by the study of
$q$-deformations of Witt and Virasoro algebras in mathematical physics
\cite{Hu1}. The key feature of these structures is the presence of a
twisting linear map $\alpha_{\mathfrak{g}}$ that deforms the Jacobi identity;
when this map is the identity, one recovers the classical Lie algebra framework.
Since their introduction, Hom-Lie algebras have been extensively studied,
encompassing representation theory \cite{Ammar1,Makhlouf1}, cohomology,
deformation theory, categorical interpretations \cite{Sheng1},
connections with discrete and deformed vector fields \cite{Larsson1},
quantization \cite{Yau1}, and universal enveloping algebras \cite{Laurent2}.
Various geometric and algebraic extensions were further developed in
\cite{Cai1,Laurent1,Mandal1,Sahin1,Merati1}.

Due to the close relationship between Lie algebras and Lie groups, a
corresponding notion of Hom-Lie groups was introduced in \cite{Jian}, together
with associated Hom-Lie algebras and one-parameter Hom-Lie subgroups. Building on
this, geometric structures on Hom-Lie groups and algebras have been actively
investigated, including Para-Sasakian geometry, almost contact and Sasakian
structures, para-K\"{a}hler, coK\"{a}hler, and cosymplectic structures
\cite{Peyg1,Peyg2,Peyg3,Peyg4,Nour2,pey1}.

A central role in this programme is played by pseudo-Riemannian Hom-Lie algebras,
which generalize to the Hom-Lie setting the pseudo-Riemannian Lie algebras
introduced by Boucetta \cite{Bouc} in the study of linear Poisson structures on
duals of Lie algebras. Killing and conformal vectors in this context were first
defined in \cite{Peyg5}, in the study of (almost) Ricci solitons on
Lorentzian--Sasakian Hom-Lie groups.

The aim of the present paper is to investigate Killing vectors in
pseudo-Riemannian Hom-Lie algebras and to establish Bochner-type theorems
relating curvature to the existence of such vectors, following the framework
developed in \cite{Peyg4,Peyg3,Peyg1,Peyg2,Nour2,Peyg5,pey1,Peyg8}. 
Specifically, we study the Hom-Levi-Civita connection, the curvature tensor,
Ricci curvature, and sectional curvature, and prove Hom-analogues of the
classical Bochner theorem in both the Riemannian and Lorentzian cases.

\medskip\noindent\textbf{Organization of the paper.}
Section~2 recalls the main definitions and preliminary notions used throughout.
Section~3 establishes Hom-versions of the classical Bochner theorem in both the
Riemannian and Lorentzian settings.
Section~4 proves that the space of Killing vectors on a pseudo-Riemannian
Hom-Lie algebra forms a totally geodesic Hom-Lie subalgebra.

\section{Preliminaries on Hom-Lie algebra and pseudo-Riemannian Hom-Lie algebra}

In this section, we recall the fundamental definitions and properties of 
Hom-Lie algebras and pseudo-Riemannian Hom-Lie algebras that will be used 
throughout the paper.

A Hom-Lie algebra is a triple
\[
(\mathfrak{g}, [\cdot,\cdot]_{\mathfrak{g}}, \alpha_{\mathfrak{g}})
\]
consisting of a vector space $\mathfrak{g}$, a skew-symmetric bilinear map 
(the bracket)
\[
[\cdot,\cdot]_{\mathfrak{g}} : \mathfrak{g} \times \mathfrak{g} \to \mathfrak{g},
\]
and a linear map (algebra morphism) $\alpha_{\mathfrak{g}} : \mathfrak{g} \to 
\mathfrak{g}$ such that the Hom-Jacobi identity holds:
\[
[\alpha_{\mathfrak{g}}(x), [y,z]_{\mathfrak{g}}]_{\mathfrak{g}}
+ [\alpha_{\mathfrak{g}}(y), [z,x]_{\mathfrak{g}}]_{\mathfrak{g}}
+ [\alpha_{\mathfrak{g}}(z), [x,y]_{\mathfrak{g}}]_{\mathfrak{g}}
= 0, \quad \forall\, x,y,z \in \mathfrak{g}.
\]

For more details on Hom-Lie algebras, see \cite{Peyg4,Peyg1,Peyg8}.
\begin{remark}
By definition, if $( \mathfrak{g}; [. ; . ]_{ \mathfrak{g}}; \alpha_{ \mathfrak{g}})$ is Hom-Lie algebra, then 
\begin{equation}
\label{bracket 1}
\alpha_{ \mathfrak{g}}([x, y]_{ \mathfrak{g}}) = [\alpha_{ \mathfrak{g}}(x), \alpha_{ \mathfrak{g}}(y)], \forall \quad x, y \in  \mathfrak{g}
\end{equation}
\end{remark}
An important remark is that when $\alpha_{\mathfrak{g}} = \mathrm{Id}_{\mathfrak{g}}$, 
the Hom-Lie structure reduces to an ordinary Lie algebra, since the Hom-Jacobi 
identity becomes the classical Jacobi identity. The twisting map $\alpha_{\mathfrak{g}}$ 
thus plays the role of a deformation parameter, and much of the richness of the 
theory stems from the interplay between $\alpha_{\mathfrak{g}}$ and the bracket.

A central class of Hom-Lie algebras, particularly amenable to geometric 
investigations, is given by those for which the twisting map is not only invertible 
but also involutive. A Hom-Lie algebra is said to be regular (involutive) if 
the twisting map $\alpha_{\mathfrak{g}}$ is invertible ($\alpha_{ \mathfrak{g}}^2 = Id_{ \mathfrak{g}}$).
The involutive condition ensures that $\alpha_{\mathfrak{g}}$ acts as a linear involution on 
$\mathfrak{g}$, which greatly simplifies the algebraic and geometric structure. 
In particular, regularity allows one to define an associated Lie bracket in a 
canonical way (see Remark~\ref{rem:lie-bracket} below), and to develop a coherent 
theory of connections and curvature in the spirit of pseudo-Riemannian geometry.

A subspace $\mathfrak{h} \subseteq \mathfrak{g}$ is called a Hom-Lie subalgebra 
if it is stable under both the twisting map and the bracket, i.e.,
\[
\alpha_{\mathfrak{g}}(\mathfrak{h}) \subseteq \mathfrak{h},
\quad \text{and} \quad
[x,y]_{\mathfrak{g}} \in \mathfrak{h}, \quad \forall\, x,y \in \mathfrak{h}.
\]
In this case,
\[
\left(\mathfrak{h},\, [\cdot,\cdot]_{\mathfrak{g}}\big|_{\mathfrak{h}\times\mathfrak{h}},\, 
\alpha_{\mathfrak{g}}\big|_{\mathfrak{h}}\right)
\]
is again a Hom-Lie algebra. The stability condition on $\mathfrak{h}$ is thus 
two-fold: the subspace must be closed under the deformed bracket as in the 
classical case, but also invariant under the structural map $\alpha_{\mathfrak{g}}$. 
Throughout the sequel, we assume that all Hom-Lie algebras are regular.

\begin{remark}\label{rem:lie-bracket}
Let $(\mathfrak{g}, [\cdot,\cdot]_{\mathfrak{g}}, \alpha_{\mathfrak{g}})$ be a 
regular Hom-Lie algebra. Then the bilinear map defined by
\begin{equation}\label{bracket2}
[x,y]_{\mathrm{Lie}} 
:= 
\big[\alpha_{\mathfrak{g}}^{-1}(x),\, \alpha_{\mathfrak{g}}^{-1}(y)\big]_{\mathfrak{g}}=\alpha_{ \mathfrak{g}}^{-1}([x, y]_{ \mathfrak{g}}),
\quad \forall\, x,y \in \mathfrak{g},
\end{equation}
is a Lie bracket on $\mathfrak{g}$. Consequently, $(\mathfrak{g}, 
[\cdot,\cdot]_{\mathrm{Lie}})$ is a Lie algebra. This observation establishes 
a bridge between the Hom-Lie and classical Lie settings, and will be used 
repeatedly to transfer results between the two frameworks.
\end{remark}

We now turn to the notion of a metric structure compatible with the Hom-Lie 
algebra. The pseudo-Riemannian geometry of Lie algebras in particular, the 
study of left-invariant metrics on Lie groups has a long and rich history. 
Extending this framework to Hom-Lie algebras requires an additional compatibility 
condition between the metric and the twisting map, which ensures that $\alpha_{\mathfrak{g}}$ 
acts as an isometry with respect to the inner product.

\begin{definition}[\cite{Peyg5,Peyg3, Peyg6}]
Let $(\mathfrak{g}, [\cdot,\cdot]_{\mathfrak{g}}, \alpha_{\mathfrak{g}})$ be a 
finite-dimensional Hom-Lie algebra. A pseudo-Riemannian metric on $\mathfrak{g}$ 
is a bilinear, symmetric, and non-degenerate form
\[
\langle \cdot , \cdot \rangle : \mathfrak{g} \times \mathfrak{g} \to \mathbb{R}
\]
satisfying, for all $x,y \in \mathfrak{g}$,
\begin{equation}\label{isom}
\langle \alpha_{\mathfrak{g}}(x), \alpha_{\mathfrak{g}}(y) \rangle
=
\langle x , y \rangle.
\end{equation}
In this case, the quadruple
\[
(\mathfrak{g}, [\cdot,\cdot]_{\mathfrak{g}}, \alpha_{\mathfrak{g}}, \langle \cdot , \cdot \rangle)
\]
is called a pseudo-Riemannian Hom-Lie algebra.
\end{definition}

Condition~\eqref{isom} states precisely that $\alpha_{\mathfrak{g}}$ is an isometry 
of $(\mathfrak{g}, \langle \cdot, \cdot \rangle)$. This is the natural 
compatibility requirement between the metric and the algebraic structure, 
analogous to the condition that the Killing form of a semisimple Lie algebra 
is invariant under automorphisms. When $\langle \cdot, \cdot \rangle$ is 
positive definite, one obtains a Riemannian Hom-Lie algebra; in general, the 
signature may be indefinite, giving rise to Lorentzian when the signature is $(1, n-1)$ or more general 
pseudo-Riemannian structures.

\begin{remark}
Since $\alpha_{\mathfrak{g}}$ is invertible, condition~\eqref{isom} also implies 
that $\alpha_{\mathfrak{g}}^{-1}$ preserves the scalar product. Indeed, using 
\eqref{isom}, we have for all $x,y \in \mathfrak{g}$,
\[
\left\langle \alpha_{\mathfrak{g}}(x), \alpha_{\mathfrak{g}}(y) \right\rangle
=
\left\langle x,y \right\rangle.
\]
Replacing $x$ by $\alpha_{\mathfrak{g}}^{-1}(x)$ and $y$ by 
$\alpha_{\mathfrak{g}}^{-1}(y)$, we obtain
\[
\left\langle \alpha_{\mathfrak{g}}(\alpha_{\mathfrak{g}}^{-1}(x)),\, 
\alpha_{\mathfrak{g}}(\alpha_{\mathfrak{g}}^{-1}(y)) \right\rangle
=
\left\langle \alpha_{\mathfrak{g}}^{-1}(x),\, \alpha_{\mathfrak{g}}^{-1}(y) \right\rangle.
\]
Since $\alpha_{\mathfrak{g}} \circ \alpha_{\mathfrak{g}}^{-1} = \mathrm{Id}_{\mathfrak{g}}$, 
this yields
\begin{equation}
\label{isom2}
\left\langle x, y \right\rangle
=
\left\langle \alpha_{\mathfrak{g}}^{-1}(x),\, \alpha_{\mathfrak{g}}^{-1}(y) \right\rangle.
\end{equation}
Hence, both $\alpha_{\mathfrak{g}}$ and $\alpha_{\mathfrak{g}}^{-1}$ preserve 
the scalar product $\langle \cdot , \cdot \rangle$. This symmetric role of 
$\alpha_{\mathfrak{g}}$ and its inverse is a distinctive feature of the regular 
setting, and will prove essential in the study of the Hom-Levi-Civita connection.
\end{remark}

Once a compatible metric is fixed, one can develop a theory of connections in 
the Hom-Lie algebra setting. The appropriate notion of a torsion-free, 
metric-compatible connection is given by the following.

\cite{Peyg4, pey1,Peyg6} A Hom-Levi-Civita connection on a pseudo-Riemannian Hom-Lie algebra
\[
(\mathfrak{g}, [\cdot,\cdot]_{\mathfrak{g}}, \alpha_{\mathfrak{g}}, \langle \cdot , \cdot \rangle)
\]
is a unique connection $\nabla$ on $\mathfrak{g}$ given by the Koszul formula:
\begin{equation}\label{koszul}
2\left\langle \nabla_x y ,\, \alpha_{\mathfrak{g}}(z) \right\rangle
=
\left\langle [x,y]_{\mathfrak{g}}, \alpha_{\mathfrak{g}}(z) \right\rangle
+ \left\langle [z,y]_{\mathfrak{g}}, \alpha_{\mathfrak{g}}(x) \right\rangle
+ \left\langle [z,x]_{\mathfrak{g}}, \alpha_{\mathfrak{g}}(y) \right\rangle,
\end{equation}
for all $x,y,z \in \mathfrak{g}$, together with:
\begin{equation}\label{torsion-free}
[x,y]_{\mathfrak{g}} = \nabla_x y - \nabla_y x,
\end{equation}
\begin{equation}\label{compa}
\left\langle \nabla_x y ,\, \alpha_{\mathfrak{g}}(z) \right\rangle
=
- \left\langle \alpha_{\mathfrak{g}}(y),\, \nabla_x z \right\rangle.
\end{equation}
Equation~\eqref{torsion-free} is the torsion-free condition, which in the 
Hom-Lie setting takes the same form as in classical Riemannian geometry, while 
\eqref{compa} expresses the metric-compatibility of $\nabla$ with respect 
to $\alpha_{\mathfrak{g}}$. The Koszul formula \eqref{koszul} uniquely characterizes 
$\nabla$ as the analog of the classical Levi-Civita connection.

Moreover, the Hom-Levi-Civita connection satisfies the compatibility condition
\begin{equation}\label{preserv}
\alpha_{\mathfrak{g}}(\nabla_x y)
=
\nabla_{\alpha_{\mathfrak{g}}(x)} \alpha_{\mathfrak{g}}(y),
\end{equation}
from which we deduce
\begin{equation}\label{preserv2}
\alpha_{\mathfrak{g}}^{-1}(\nabla_x y)
=
\nabla_{\alpha_{\mathfrak{g}}^{-1}(x)} \alpha_{\mathfrak{g}}^{-1}(y).
\end{equation}
These equivariance properties show that $\nabla$ intertwines the action of 
$\alpha_{\mathfrak{g}}$ on $\mathfrak{g}$, and are the Hom-algebraic counterpart 
of the invariance of the Levi-Civita connection under isometries in classical 
geometry.

With the connection at hand, one defines the curvature tensor in the standard 
fashion, adapted to the Hom-Lie setting. The curvature tensor $R$ of $\nabla$ 
on a pseudo-Riemannian Hom-Lie algebra is defined by
\begin{equation}
R(x,y)z
=
\nabla_{\alpha_{\mathfrak{g}}(x)} \nabla_y z
- \nabla_{\alpha_{\mathfrak{g}}(y)} \nabla_x z
- \nabla_{[x,y]_{\mathfrak{g}}} \alpha_{\mathfrak{g}}(z),
\end{equation}
for all $x,y,z \in \mathfrak{g}$ (see \cite{Peyg6, Peyg5}). The presence of $\alpha_{\mathfrak{g}}$ in 
the definition reflects the twisted nature of the Hom-Jacobi identity and 
ensures that $R$ is compatible with the algebraic structure of $\mathfrak{g}$.

The curvature tensor is the fundamental invariant encoding the local geometry 
of the space. From it, one derives the classical notions of Ricci and scalar 
curvatures. Following \cite{Peyg4, Nour2,Peyg6}, the Ricci curvature tensor and scalar 
curvature are defined respectively by
\begin{equation}
\mathrm{Ric}(x,y) := \mathrm{tr}\big(z \mapsto R(x,y)z\big),
\end{equation}
\begin{equation}
r := \mathrm{tr}(\mathrm{Ric}).
\end{equation}
These quantities play a central role in the 
classification of pseudo-Riemannian Hom-Lie algebras and in the formulation 
of Einstein-type equations in this context.

Finally, the sectional curvature associated with a two-dimensional non degenerate subspace 
spanned by $x,y \in \mathfrak{g}$ is given by
\begin{equation}\label{sectio}
K(x,y)
=
\frac{
\left\langle R(x,y)y ,\, \alpha_{\mathfrak{g}}^2(x) \right\rangle
}{
\|x\|^2 \|y\|^2 - \left\langle x , y \right\rangle^2
},
\end{equation}
where $\|x\|^2 = \left\langle x, x \right\rangle$ (see\cite{Peyg4, Peyg5}). 
Note that $K(x,y) = K(y,x).$ This follows from properties of the curvature tensor $R$ (see \cite[Page 15]{Peyg4}).
\section{Bochner-Type Theorems}

Killing and conformal vectors in pseudo-Riemannian Hom-Lie algebras were
introduced in \cite{Peyg5} in the study of (almost) Ricci solitons on
Lorentzian--Sasakian Hom-Lie groups. We investigate here further properties
of such vectors, establishing in particular a curvature identity that leads
to several Bochner-type theorems in both the Riemannian and
Lorentzian settings.

\begin{definition}[\cite{Peyg5}]
Let $(\mathfrak{g};\,[.\,;\,.]_{\mathfrak{g}};\,\alpha_{\mathfrak{g}},\langle\,.\,,\,.\,\rangle)$
be a pseudo-Riemannian Hom-Lie algebra. A vector $v \in \mathfrak{g}$ is
called \emph{conformal} if there exists a real scalar $\rho$ such that for
all $x, y \in \mathfrak{g}$,
\begin{equation}
  \langle \nabla_x v,\,\alpha_{\mathfrak{g}}(y) \rangle
  + \langle \nabla_y v,\,\alpha_{\mathfrak{g}}(x) \rangle
  = 2\rho\,\langle x,\,\alpha_{\mathfrak{g}}^2(y) \rangle.
\end{equation}
The vector $v$ is said to be \emph{Killing} if $\rho = 0$, i.e.,
\begin{equation}
\label{kill}
  \langle \nabla_x v,\,\alpha_{\mathfrak{g}}(y) \rangle
  = -\langle \nabla_y v,\,\alpha_{\mathfrak{g}}(x) \rangle
\end{equation}
for all $x, y \in \mathfrak{g}$.
\end{definition}

We prove the following result, which will be used throughout the paper.

\begin{lemma}
\label{lem kil}
Let $(\mathfrak{g};\,[.\,;\,.]_{\mathfrak{g}};\,\alpha_{\mathfrak{g}},\langle\,.\,,\,.\,\rangle)$
be a pseudo-Riemannian Hom-Lie algebra and let $v$ be a Killing vector. Then:
\begin{itemize}
  \item[i)]  $\nabla_v v = 0$;
  \item[ii)] $\alpha_{\mathfrak{g}}(v)$ and $\alpha_{\mathfrak{g}}^{-1}(v)$ are
             also Killing vectors.
\end{itemize}
\end{lemma}

\begin{proof}
i) Since $v$ is Killing, setting $x = v$ in~\eqref{kill} gives
\[
  \langle \nabla_v v,\,\alpha_{\mathfrak{g}}(y) \rangle
  = -\langle \nabla_y v,\,\alpha_{\mathfrak{g}}(v) \rangle
\]
for all $y \in \mathfrak{g}$. From~\eqref{compa} we have
$\langle \nabla_y v,\,\alpha_{\mathfrak{g}}(v) \rangle = 0$, so
$\langle \nabla_v v,\,\alpha_{\mathfrak{g}}(y) \rangle = 0$ for all
$y \in \mathfrak{g}$. Since $\alpha_{\mathfrak{g}}$ is invertible, this is
equivalent to $\langle \nabla_v v,\,z \rangle = 0$ for all $z \in \mathfrak{g}$,
hence $\nabla_v v = 0$.

ii) We prove only that $\alpha_{\mathfrak{g}}(v)$ is Killing; the argument for
$\alpha_{\mathfrak{g}}^{-1}(v)$ is analogous. For all $x, y \in \mathfrak{g}$,
using~\eqref{isom2} and~\eqref{preserv2} we have
\[
  \langle \nabla_x \alpha_{\mathfrak{g}}(v),\,\alpha_{\mathfrak{g}}(y) \rangle
  = \langle \nabla_{\alpha_{\mathfrak{g}}^{-1}(x)} v,\,y \rangle.
\]
Writing $y = \alpha_{\mathfrak{g}}(\alpha_{\mathfrak{g}}^{-1}(y))$ and applying
the Killing condition for $v$ yields
\[
  \langle \nabla_{\alpha_{\mathfrak{g}}^{-1}(x)} v,\,y \rangle
  = -\langle \nabla_{\alpha_{\mathfrak{g}}^{-1}(y)} v,\,x \rangle
  = -\langle \nabla_y \alpha_{\mathfrak{g}}(v),\,\alpha_{\mathfrak{g}}(x) \rangle,
\]
which proves that $\alpha_{\mathfrak{g}}(v)$ is Killing.
\end{proof}

Before stating the main theorems, we illustrate the influence of curvature on
the existence of Killing vector fields. The following result will be essential
in the sequel.

\begin{proposition}
Let $(\mathfrak{g};\,[.\,;\,.]_{\mathfrak{g}};\,\alpha_{\mathfrak{g}},\langle\,.\,,\,.\,\rangle)$
be a pseudo-Riemannian Hom-Lie algebra and let $v$ be a Killing vector. Then
for any $x \in \mathfrak{g}$,
\begin{equation}
\label{import}
  \langle R(x,v)v,\,\alpha_{\mathfrak{g}}^2(x) \rangle
  = \langle \nabla_x v,\,\nabla_x v \rangle.
\end{equation}
In particular, if $x$ and $v$ span a non degenerate plane, then
\begin{equation}
\label{import2}
K(x,v)
=
\frac{\langle \nabla_x v,\nabla_x v \rangle}
{\|x\|^2\|v\|^2 - \langle x,v \rangle^2}
\end{equation}
\end{proposition}

\begin{proof}
We have
$R(x,v)v = \nabla_{\alpha_{\mathfrak{g}}(x)}\nabla_v v
           - \nabla_{\alpha_{\mathfrak{g}}(v)}\nabla_x v
           - \nabla_{[x,v]_{\mathfrak{g}}}\alpha_{\mathfrak{g}}(v)$.
By Lemma~\ref{lem kil}(i), $\nabla_v v = 0$, so
\[
  R(x,v)v = -\nabla_{\alpha_{\mathfrak{g}}(v)}\nabla_x v
             - \nabla_{[x,v]_{\mathfrak{g}}}\alpha_{\mathfrak{g}}(v),
\]
and therefore
\[
  \langle R(x,v)v,\,\alpha_{\mathfrak{g}}^2(x) \rangle
  = -\langle \nabla_{\alpha_{\mathfrak{g}}(v)}\nabla_x v,\,\alpha_{\mathfrak{g}}^2(x) \rangle
    -\langle \nabla_{[x,v]_{\mathfrak{g}}}\alpha_{\mathfrak{g}}(v),\,\alpha_{\mathfrak{g}}^2(x) \rangle.
\]
We treat each term separately.

\medskip\noindent\textit{First term.}
Using~\eqref{isom2} and~\eqref{preserv2},
\[
  \langle \nabla_{\alpha_{\mathfrak{g}}(v)}\nabla_x v,\,\alpha_{\mathfrak{g}}^2(x) \rangle
  = \langle \nabla_v \nabla_{\alpha_{\mathfrak{g}}^{-1}(x)}\alpha_{\mathfrak{g}}^{-1}(v),\,\alpha_{\mathfrak{g}}(x) \rangle.
\]
Applying~\eqref{compa} and~\eqref{preserv} then gives
\[
  \langle \nabla_v \nabla_{\alpha_{\mathfrak{g}}^{-1}(x)}\alpha_{\mathfrak{g}}^{-1}(v),\,\alpha_{\mathfrak{g}}(x) \rangle
  = -\langle \nabla_x v,\,\nabla_v x \rangle,
\]
so the contribution of the first term is $\langle \nabla_x v,\,\nabla_v x \rangle$.

\medskip\noindent\textit{Second term.}
Using~\eqref{isom2}, \eqref{preserv2}, and~\eqref{bracket2},
\[
  \langle \nabla_{[x,v]_{\mathfrak{g}}}\alpha_{\mathfrak{g}}(v),\,\alpha_{\mathfrak{g}}^2(x) \rangle
  = \langle \nabla_{[\alpha_{\mathfrak{g}}^{-1}(x),\alpha_{\mathfrak{g}}^{-1}(v)]} v,\,\alpha_{\mathfrak{g}}(x) \rangle.
\]
Since $v$ is Killing and by~\eqref{bracket 1},
\[
  \langle \nabla_{[\alpha_{\mathfrak{g}}^{-1}(x),\alpha_{\mathfrak{g}}^{-1}(v)]} v,\,\alpha_{\mathfrak{g}}(x) \rangle
  = -\langle \nabla_x v,\,[x,v]_{\mathfrak{g}} \rangle.
\]
From~\eqref{torsion-free}, $[x,v]_{\mathfrak{g}} = \nabla_x v - \nabla_v x$, so the
contribution of the second term is
$\langle \nabla_x v,\,\nabla_x v \rangle - \langle \nabla_x v,\,\nabla_v x \rangle$.

\medskip
Summing the two contributions yields~\eqref{import}. Equality~\eqref{import2}
follows from~\eqref{sectio}.
\end{proof}

We recall the classical Bochner theorem on Killing vector fields in Riemannian
geometry.

\begin{theorem}[\cite{Boch}]
Let $M$ be a compact Riemannian manifold with non-positive Ricci curvature.
Then every Killing vector field on $M$ is parallel. Moreover, if the Ricci
curvature is negative, then there exists no non-trivial Killing vector field.
\end{theorem}

In the spirit of this classical result, we now establish its analogue in the
Hom-Lie setting.

\begin{theorem}[Bochner theorem for Hom-Lie algebras]
\label{Boch Hom}
Let $(\mathfrak{g};\,[.\,;\,.]_{\mathfrak{g}};\,\alpha_{\mathfrak{g}},\langle\,.\,,\,.\,\rangle)$
be a Riemannian Hom-Lie algebra.
\begin{itemize}
  \item[i)]  If $\mathfrak{g}$ has negative sectional curvature, then it admits
             no non-trivial Killing vector.
  \item[ii)] If $\mathfrak{g}$ has non-positive sectional curvature, then every
             Killing vector $v$ is parallel, i.e., $\nabla_x v = 0$ for all
             $x \in \mathfrak{g}$.
\end{itemize}
\end{theorem}

\begin{proof}
i) Suppose there exists a non-trivial Killing vector $v$ and choose $x$ not
proportional to $v$. Then $K(x,v) < 0$, and since $\langle\,.\,,\,.\,\rangle$
is positive definite, $(\|x\|^2\|v\|^2 - \langle x,v \rangle^2) > 0$.
From~\eqref{import2} we get $\langle \nabla_x v,\,\nabla_x v \rangle < 0$,
contradicting positive definiteness.

ii) For any Killing vector $v$, equation~\eqref{import2} and the non-positive
sectional curvature give $\langle \nabla_x v,\,\nabla_x v \rangle \leq 0$ if $x$ is not proportional to $v$.
Since $\langle\,.\,,\,.\,\rangle$ is positive definite, we obtain $\nabla_x v = 0$. In the case where $x=\lambda v$, Lemma~\ref{lem kil} implies that
\[
\nabla_x v=\lambda \nabla_v v=0.
\]
\end{proof}

For involutive Hom-Lie algebras, the sectional curvature condition can be
replaced by a Ricci curvature condition.

\begin{theorem}[Bochner theorem for involutive Hom-Lie algebras]
\label{Boch inv}
Let $(\mathfrak{g};\,[.\,;\,.]_{\mathfrak{g}};\,\alpha_{\mathfrak{g}},\langle\,.\,,\,.\,\rangle)$
be an involutive Riemannian Hom-Lie algebra.
\begin{itemize}
  \item[i)]  If $\mathfrak{g}$ has negative Ricci curvature, then it admits no
             non-trivial Killing vector.
  \item[ii)] If $\mathfrak{g}$ has non-positive Ricci curvature, then every
             Killing vector $v$ is parallel, i.e., $\nabla_x v = 0$ for all
             $x \in \mathfrak{g}$.
\end{itemize}
\end{theorem}

\begin{proof}
Since $\mathfrak{g}$ is involutive, $\alpha_{\mathfrak{g}}^2(x) = x$, and~\eqref{import}
becomes
\[
  \langle R(x,v)v,\,x \rangle = \langle \nabla_x v,\,\nabla_x v \rangle,
  \quad \forall\, x \in \mathfrak{g}.
\]
Taking an orthonormal basis $(e_i)_{1 \leq i \leq n}$, we get
\[
  \mathrm{Ric}(v,v) = \sum_{i=1}^n \langle \nabla_{e_i} v,\,\nabla_{e_i} v \rangle.
\]
If (i) holds and there exists a non-trivial Killing vector, then
$\mathrm{Ric}(v,v) < 0$ whereas $\sum_{i=1}^n \langle \nabla_{e_i} v,\,\nabla_{e_i} v \rangle \geq 0$,
a contradiction.

If (ii) holds and there exists a Killing vector, then
$\sum_{i=1}^n \langle \nabla_{e_i} v,\,\nabla_{e_i} v \rangle \leq 0$.
Since $\langle\,.\,,\,.\,\rangle$ is positive definite, we get $\nabla_{e_i} v = 0$
for each $e_i$, and by linearity $\nabla_x v = 0$ for all $x \in \mathfrak{g}$.
\end{proof}

In the Lorentzian setting, the analogous of Bochner theorem  were obtained in \cite{Romeo} for timelike Killing vector field (see Example 3.5, Corollary 3.8, proof of Corollary 3.10). We can summarize the results in the following.
\begin{theorem}
    let M be a compact lorentzian manifold, such that $Ric(X,X)<0 $ for all timelike vector X, then there exists no timelike Killing vector field.
If  $Ric(X,X) \leq 0$ for all timelike vector $X$,
Then any timelike Killing vector field is parallel.

\end{theorem}
Now, we prove the following.
\begin{theorem}
\label{Boch Lor}
Let $(\mathfrak{g};\,[.\,;\,.]_{\mathfrak{g}};\,\alpha_{\mathfrak{g}},\langle\,.\,,\,.\,\rangle)$
be a Lorentzian Hom-Lie algebra.
\begin{itemize}
  \item[i)]  If $\mathfrak{g}$ has positive sectional curvature on timelike
             planes, then it admits no non-trivial timelike Killing vector.
  \item[ii)] If $\mathfrak{g}$ has non-positive sectional curvature on timelike
             planes, then every timelike Killing vector $v$ is parallel, i.e.,
             $\nabla_x v = 0$ for all $x \in \mathfrak{g}$.
\end{itemize}
\end{theorem}

\begin{proof}
i) Suppose there exists a non-trivial timelike Killing vector $v$. From~\eqref{compa},
$\langle \nabla_x v,\,\alpha_{\mathfrak{g}}(v) \rangle = 0$ for all $x \in \mathfrak{g}$.
Since $v$ is timelike, $\alpha_{\mathfrak{g}}(v)$ is also timelike by~\eqref{isom}.
As the timelike vector $\alpha_{\mathfrak{g}}(v)$ is orthogonal to $\nabla_x v$,
we conclude that $\nabla_x v$ is spacelike for all $x \in \mathfrak{g}$, i.e.,
$\langle \nabla_x v,\,\nabla_x v \rangle \geq 0$. Now choose $x$ not proportional
to $v$. By hypothesis $K(x,v) > 0$, and since $x$ and $v$ span a timelike plane,
$\|x\|^2\|v\|^2 - \langle x,v \rangle^2 < 0$. From~\eqref{import2} we get
$\langle \nabla_x v,\,\nabla_x v \rangle < 0$, a contradiction.

ii) Suppose there exists a non-trivial timelike Killing vector $v$. The argument
in (i) shows that $\langle \nabla_x v,\,\nabla_x v \rangle \geq 0$ for all
$x \in \mathfrak{g}$. If $x = \lambda v$, then $\nabla_x v = \lambda \nabla_v v = 0$
by Lemma~\ref{lem kil}. If $x$ is not proportional to $v$, then $x$ and $v$ span
a timelike plane, so $\|x\|^2\|v\|^2 - \langle x,v \rangle^2 < 0$ and
$K(x,v) \geq 0$. From~\eqref{import2}, $\langle \nabla_x v,\,\nabla_x v \rangle \leq 0$.
Combined with the reverse inequality, we get $\langle \nabla_x v,\,\nabla_x v \rangle = 0$,
and since $\nabla_x v$ is spacelike, $\nabla_x v = 0$ for all $x \in \mathfrak{g}$.
\end{proof}

For involutive Lorentzian Hom-Lie algebras, we obtain the following Ricci-curvature
version.

\begin{theorem}
\label{Boch inv Lor}
Let $(\mathfrak{g};\,[.\,;\,.]_{\mathfrak{g}};\,\alpha_{\mathfrak{g}},\langle\,.\,,\,.\,\rangle)$
be an involutive Lorentzian Hom-Lie algebra.
\begin{itemize}
  \item[i)]  If $\mathfrak{g}$ has negative Ricci curvature on timelike vectors
             $(\mathrm{Ric}(x,x) < 0)$, then it admits no non-trivial timelike
             Killing vector.
  \item[ii)] If $\mathfrak{g}$ has non-positive Ricci curvature on timelike
             vectors, then every timelike Killing vector $v$ is parallel, i.e.,
             $\nabla_x v = 0$ for all $x \in \mathfrak{g}$.
\end{itemize}
\end{theorem}

\begin{proof}
Since $\mathfrak{g}$ is involutive, $\alpha_{\mathfrak{g}}^2(x) = x$, and~\eqref{import}
becomes
\[
  \langle R(x,v)v,\,x \rangle = \langle \nabla_x v,\,\nabla_x v \rangle,
  \quad \forall\, x \in \mathfrak{g}.
\]
Taking an orthonormal basis $(e_i)_{1 \leq i \leq n}$ with $e_1 = v/|v|$, so
that $e_i$ is spacelike for $2 \leq i \leq n$, we get
\[
  \mathrm{Ric}(v,v) = \sum_{i=2}^n \langle \nabla_{e_i} v,\,\nabla_{e_i} v \rangle.
\]
Since $\nabla_{e_i} v$ is spacelike, $\langle \nabla_{e_i} v,\,\nabla_{e_i} v \rangle \geq 0$
for $2 \leq i \leq n$.

If (i) holds and there exists a non-trivial timelike Killing vector, then
$\mathrm{Ric}(v,v) < 0$ whereas
$\sum_{i=2}^n \langle \nabla_{e_i} v,\,\nabla_{e_i} v \rangle \geq 0$,
a contradiction.

If (ii) holds and there exists a timelike Killing vector, then
$\mathrm{Ric}(v,v) \leq 0$, so
$\sum_{i=2}^n \langle \nabla_{e_i} v,\,\nabla_{e_i} v \rangle \leq 0$.
Together with the non-negativity of each term, this forces
$\langle \nabla_{e_i} v,\,\nabla_{e_i} v \rangle = 0$, and since
$\nabla_{e_i} v$ is spacelike, $\nabla_{e_i} v = 0$ for each $i \geq 2$.
Furthermore, $e_1 = v/|v|$, so $\nabla_{e_1} v = (1/|v|)\nabla_v v = 0$
by Lemma~\ref{lem kil}. By linearity, $\nabla_x v = 0$ for all
$x \in \mathfrak{g}$.
\end{proof}

\section{Structure of the Hom-Lie Subalgebra of Killing Vectors}

To establish that Killing vectors form a Hom-Lie subalgebra, we first prove
the following identity for the Lie derivative of the pseudo-Riemannian metric.
Following \cite[relation 2.2]{Peyg1}, the Lie derivative of the metric is given by:
$L_w\langle, \rangle(u,v) = -\langle [w,u],  \alpha_g(v)\rangle - \langle \alpha_g(u), [w,v] \rangle.$

Then, it easy to see that $w$ is a Killing vector if and only if $L_w\langle, \rangle = 0.$

\begin{proposition}

Let $(\mathfrak{g};\,[.\,;\,.]_{\mathfrak{g}};\,\alpha_{\mathfrak{g}},\langle\,.\,,\,.\,\rangle)$
be a pseudo-Riemannian Hom-Lie algebra and let $v, w \in \mathfrak{g}$. Then
\begin{equation}
\label{Lie deriv metric}
  L_w\bigl(L_{\alpha_{\mathfrak{g}}(v)}\langle\,.\,,\,.\,\rangle\bigr)
  - L_v\bigl(L_{\alpha_{\mathfrak{g}}(w)}\langle\,.\,,\,.\,\rangle\bigr)
  = L_{\alpha_{\mathfrak{g}}^{-1}([w,v])}\langle\,.\,,\,.\,\rangle.
\end{equation}
\end{proposition}

\begin{proof}
Take $x, y \in \mathfrak{g}$. We compute:
\begin{align*}
\bigl(L_w(L_{\alpha_{\mathfrak{g}}(v)}\langle\,.\,,\,.\,\rangle)\bigr)(x,y)
&= -\bigl(L_{\alpha_{\mathfrak{g}}(v)}\langle\,.\,,\,.\,\rangle\bigr)([w,x],\, \alpha_{\mathfrak{g}}(y))
   -\bigl(L_{\alpha_{\mathfrak{g}}(v)}\langle\,.\,,\,.\,\rangle\bigr)(\alpha_{\mathfrak{g}}(x),\, [w,y]) \\
&= \left\langle [\alpha_{\mathfrak{g}}(v),[w,x]],\, \alpha_{\mathfrak{g}}^2(y) \right\rangle
 + \left\langle \alpha_{\mathfrak{g}}([w,x]),\, [\alpha_{\mathfrak{g}}(v),\alpha_{\mathfrak{g}}(y)] \right\rangle \\
&\quad
 + \left\langle [\alpha_{\mathfrak{g}}(v),\alpha_{\mathfrak{g}}(x)],\, \alpha_{\mathfrak{g}}([w,y]) \right\rangle
 + \left\langle \alpha_{\mathfrak{g}}^2(x),\, [\alpha_{\mathfrak{g}}(v),[w,y]] \right\rangle \\
&= \left\langle [\alpha_{\mathfrak{g}}(v),[w,x]],\, \alpha_{\mathfrak{g}}^2(y) \right\rangle
 + \left\langle [w,x],\, [v,y] \right\rangle \\
&\quad
 + \left\langle [v,x],\, [w,y] \right\rangle
 + \left\langle \alpha_{\mathfrak{g}}^2(x),\, [\alpha_{\mathfrak{g}}(v),[w,y]] \right\rangle,
\end{align*}
where we used in the last step the fact that $\alpha_{\mathfrak{g}}$ preserves
both the bracket and the scalar product. Similarly,
\begin{align*}
\bigl(L_v(L_{\alpha_{\mathfrak{g}}(w)}\langle\,.\,,\,.\,\rangle)\bigr)(x,y)
&= -\bigl(L_{\alpha_{\mathfrak{g}}(w)}\langle\,.\,,\,.\,\rangle\bigr)([v,x],\, \alpha_{\mathfrak{g}}(y))
   -\bigl(L_{\alpha_{\mathfrak{g}}(w)}\langle\,.\,,\,.\,\rangle\bigr)(\alpha_{\mathfrak{g}}(x),\, [v,y]) \\
&= \left\langle [\alpha_{\mathfrak{g}}(w),[v,x]],\, \alpha_{\mathfrak{g}}^2(y) \right\rangle
 + \left\langle \alpha_{\mathfrak{g}}([v,x]),\, [\alpha_{\mathfrak{g}}(w),\alpha_{\mathfrak{g}}(y)] \right\rangle \\
&\quad
 + \left\langle [\alpha_{\mathfrak{g}}(w),\alpha_{\mathfrak{g}}(x)],\, \alpha_{\mathfrak{g}}([v,y]) \right\rangle
 + \left\langle \alpha_{\mathfrak{g}}^2(x),\, [\alpha_{\mathfrak{g}}(w),[v,y]] \right\rangle \\
&= \left\langle [\alpha_{\mathfrak{g}}(w),[v,x]],\, \alpha_{\mathfrak{g}}^2(y) \right\rangle
 + \left\langle [v,x],\, [w,y] \right\rangle \\
&\quad
 + \left\langle [w,x],\, [v,y] \right\rangle
 + \left\langle \alpha_{\mathfrak{g}}^2(x),\, [\alpha_{\mathfrak{g}}(w),[v,y]] \right\rangle.
\end{align*}
Setting $A = \bigl(L_w(L_{\alpha_{\mathfrak{g}}(v)}\langle\,.\,,\,.\,\rangle)\bigr)(x,y)$
and $B = \bigl(L_v(L_{\alpha_{\mathfrak{g}}(w)}\langle\,.\,,\,.\,\rangle)\bigr)(x,y)$,
we obtain:
\begin{align*}
A - B
&= \left\langle [\alpha_{\mathfrak{g}}(v),[w,x]],\, \alpha_{\mathfrak{g}}^2(y) \right\rangle
 + \left\langle [\alpha_{\mathfrak{g}}(w),[x,v]],\, \alpha_{\mathfrak{g}}^2(y) \right\rangle \\
&\quad
 + \left\langle \alpha_{\mathfrak{g}}^2(x),\, [\alpha_{\mathfrak{g}}(v),[w,y]] \right\rangle
 + \left\langle \alpha_{\mathfrak{g}}^2(x),\, [\alpha_{\mathfrak{g}}(w),[y,v]] \right\rangle \\
&= -\left\langle [\alpha_{\mathfrak{g}}(x),[v,w]],\, \alpha_{\mathfrak{g}}^2(y) \right\rangle
 - \left\langle \alpha_{\mathfrak{g}}^2(x),\, [\alpha_{\mathfrak{g}}(y),[v,w]] \right\rangle \\
&= -\left\langle [[w,v],\alpha_{\mathfrak{g}}(x)],\, \alpha_{\mathfrak{g}}^2(y) \right\rangle
 - \left\langle \alpha_{\mathfrak{g}}^2(x),\, [[w,v],\alpha_{\mathfrak{g}}(y)] \right\rangle \\
&= -\left\langle [\alpha_{\mathfrak{g}}^{-1}([w,v]),x],\, \alpha_{\mathfrak{g}}(y) \right\rangle
 - \left\langle \alpha_{\mathfrak{g}}(x),\, [\alpha_{\mathfrak{g}}^{-1}([w,v]),y] \right\rangle,
\end{align*}
where we again used the fact that $\alpha_{\mathfrak{g}}$ preserves both the
bracket and the scalar product. Hence
\[
  A - B = \bigl(L_{\alpha_{\mathfrak{g}}^{-1}([w,v])}\langle\,.\,,\,.\,\rangle\bigr)(x,y),
\]
which gives~\eqref{Lie deriv metric}.
\end{proof}

We can now state and prove the main theorem of this section.

\begin{theorem}
\label{theorem Kill}
Let $(\mathfrak{g};\,[.\,;\,.]_{\mathfrak{g}};\,\alpha_{\mathfrak{g}},\langle\,.\,,\,.\,\rangle)$
be a pseudo-Riemannian Hom-Lie algebra. The space $\mathfrak{K}$ of Killing
vectors forms a Hom-Lie subalgebra of $\mathfrak{g}$.
\end{theorem}

\begin{proof}
It is clear that any linear combination of Killing vectors is again a Killing
vector. By Lemma~\ref{lem kil}(ii), $\alpha_{\mathfrak{g}}(\mathfrak{K}) \subset \mathfrak{K}$.
It remains to show that $\mathfrak{K}$ is closed under the bracket. Let $v, w \in \mathfrak{K}$.
By Lemma~\ref{lem kil}(ii), $\alpha_{\mathfrak{g}}(v)$ and $\alpha_{\mathfrak{g}}(w)$
are also Killing vectors, so $L_{\alpha_{\mathfrak{g}}(v)}\langle\,.\,,\,.\,\rangle = 0$
and $L_{\alpha_{\mathfrak{g}}(w)}\langle\,.\,,\,.\,\rangle = 0$. From~\eqref{Lie deriv metric}
it follows that $L_{\alpha_{\mathfrak{g}}^{-1}([w,v])}\langle\,.\,,\,.\,\rangle = 0$,
i.e., $\alpha_{\mathfrak{g}}^{-1}([w,v])$ is a Killing vector. Applying
Lemma~\ref{lem kil}(ii) once more, $[w,v]$ is a Killing vector.
\end{proof}

We recall the following definition \cite[Definition~2.1]{Peyg4}.

\begin{definition}
Let $(\mathfrak{g};\,[.\,;\,.]_{\mathfrak{g}};\,\alpha_{\mathfrak{g}})$ be a
Hom-Lie algebra. A Hom-Lie subalgebra $\mathfrak{h}$ of $\mathfrak{g}$ is
called \emph{totally geodesic} if $\nabla_x y \in \mathfrak{h}$ for all
$x, y \in \mathfrak{h}$.
\end{definition}

We show that the Hom-Lie subalgebra $\mathfrak{K}$ of Killing vectors is
totally geodesic.

\begin{proposition}
Let $(\mathfrak{g};\,[.\,;\,.]_{\mathfrak{g}};\,\alpha_{\mathfrak{g}},\langle\,.\,,\,.\,\rangle)$
be a pseudo-Riemannian Hom-Lie algebra. The space $\mathfrak{K}$ of Killing
vectors is a totally geodesic Hom-Lie subalgebra of $\mathfrak{g}$.
\end{proposition}

\begin{proof}
Let $v$ and $w$ be two Killing vectors. For any $x \in \mathfrak{g}$,
\begin{align*}
  \langle \nabla_v w,\,\alpha_{\mathfrak{g}}(x) \rangle
  &= -\langle \nabla_x w,\,\alpha_{\mathfrak{g}}(v) \rangle \\
  &= \phantom{-}\langle \alpha_{\mathfrak{g}}(w),\,\nabla_x v \rangle \\
  &= -\langle \alpha_{\mathfrak{g}}(x),\,\nabla_w v \rangle,
\end{align*}
where we used successively the Killing condition for $w$, identity~\eqref{compa},
and the Killing condition for $v$. Hence $\nabla_v w = -\nabla_w v$, which
gives $[v,w]_{\mathfrak{g}} = 2\nabla_v w$. Since $[v,w]_{\mathfrak{g}}$ is a
Killing vector by Theorem~\ref{theorem Kill}, so is $\nabla_v w$. Therefore
$\mathfrak{K}$ is a totally geodesic Hom-Lie subalgebra of $\mathfrak{g}$.
\end{proof}
\begin{example}
    Let us consider $\mathfrak{sl}(2)$ with basis ${e,f,h}$:
	$$
	[h,e]=2e,\quad [h,f]=-2f,\quad [e,f]=h.
	$$
	
	Define  $\alpha_{\mathfrak{g}}$ by:
	$$
	\alpha_{\mathfrak{g}}(e) = -e,\quad
	\alpha_{\mathfrak{g}}(f)= -f,\quad
	\alpha_{\mathfrak{g}}(h) = h,
	$$
	and the  metric $\langle \cdot , \cdot \rangle $ by :
	$$
	\langle e, f \rangle  = \langle h, h\rangle = 1 ,\quad \text{and $0$ otherwise}.
	$$
	Then \[
(\mathfrak{sl}(2), [\cdot,\cdot], \alpha_{\mathfrak{g}}, \langle \cdot , \cdot \rangle)
\] is a pseudo-Riemannian Hom-Lie algebra. Moreover, the space $\mathfrak{K}$ of Killing vectors is given by $\mathfrak{K} = span (h)$.
\end{example}

\bibliographystyle{amsplain}

\bigskip

\noindent\textbf{Raymond HOUNNONKPE}\\
Institut Ouest Africain de Mathématiques, Université Gamal Abdel Nasser de Conakry, Guinée.\\
E-mail: \texttt{adivignon.hounnonkpe@ioam.uganc.edu.gn}\\
E-mail: \texttt{rhounnonkpe@ymail.com}

\medskip

\noindent\textbf{Ibrahima BAKAYOKO}\\
Université de N'Zérékoré, Guinée.\\
E-mail: \texttt{ibrahimabakayoko27@gmail.com}

\medskip

\noindent\textbf{Emmanuel GNANDI}\\
INSA Toulouse, Département de Génie Mathématique, IMT,\\
135 Avenue de Rangueil, 31077 Toulouse Cedex 4, France.\\
E-mail: \texttt{kpanteemmanuel@gmail.com}\\
E-mail: \texttt{gnandi@insa-toulouse.fr}

\medskip

\noindent\textbf{Jo\"el TOSSA}\\
Institut Ouest Africain de Mathématiques, Université Gamal Abdel Nasser de Conakry, Guinée.\\
E-mail: \texttt{joeltossa@gmail.com}
\end{document}